\documentclass[11pt,a4paper]{article}
\usepackage[left=2.0cm, top=2.45cm,bottom=2.45cm,right=2.0cm]{geometry}
\usepackage{mathtools,amssymb,amsthm,calc,graphicx,xcolor,cleveref,tikz,bm,bbm,url,enumerate}

\usepackage{amsfonts}              
\usepackage[T1]{fontenc}
\numberwithin{equation}{section}

\usepackage{subcaption}

\newtheorem {theorem}{Theorem}[section]
\newtheorem {proposition}[theorem]{Proposition}
\newtheorem {lemma}[theorem]{Lemma}

\newtheorem {remark}[theorem]{Remark}

{\theoremstyle{definition}
	
}

\newcommand{\Var}{\operatorname{Var}}

\newcommand{\conv}{\textup{conv}}

\newcommand{\vol}{\textup{vol}}

\newcommand{\dist}{\textup{dist}}

\newcommand{\ind}[1]{\mathbbm{1}_{\{#1\}}}

\def\EE{\mathbb{E}}

\def\HH{\mathbb{H}}

\def\NN{\mathbb{N}}

\def\PP{\mathbb{P}}

\def\RR{\mathbb{R}}

\def\XX{\mathbb{X}}
\def\YY{\mathbb{Y}}

\def\cA{\mathcal{A}}

\makeatletter
\let\@fnsymbol\@alph
\makeatother

\begin{document}
	
\title{\bfseries Central limit theorems for the nearest neighbour embracing graph in Euclidean and hyperbolic space}

\author{Holger Sambale\footnotemark[1],\;\; Christoph Th\"ale\footnotemark[2]\;\; and Tara Trauthwein\footnotemark[3]}

\date{}
\renewcommand{\thefootnote}{\fnsymbol{footnote}}
\footnotetext[1]{Bielefeld University, Germany. Email: hsambale@math.uni-bielefeld.de}

\footnotetext[2]{Ruhr University Bochum, Germany. Email: christoph.thaele@rub.de}

\footnotetext[3]{University of Oxford, United Kingdom. Email: tara.trauthwein@stats.ox.ac.uk}

\maketitle

\begin{abstract} 
Consider a stationary Poisson process $\eta$ in the $d$-dimensional Euclidean or hyperbolic space and construct a random graph with vertex set $\eta$ as follows. First, each point $x\in\eta$ is connected  by an edge to its nearest neighbour, then to its second nearest neighbour and so on, until $x$ is contained in the convex hull of the points already connected to $x$. The resulting random graph is the so-called nearest neighbour embracing graph. The main result of this paper is a quantitative description of the Gaussian fluctuations of geometric functionals associated with the nearest neighbour embracing graph. {More precisely, the total edge length, more general length-power functionals and the number of vertices with given outdegree are considered}.\\[2mm]
\end{abstract}

\section{Introduction}

\begin{figure}[t]
    \centering
    \includegraphics[width=.6\textwidth]{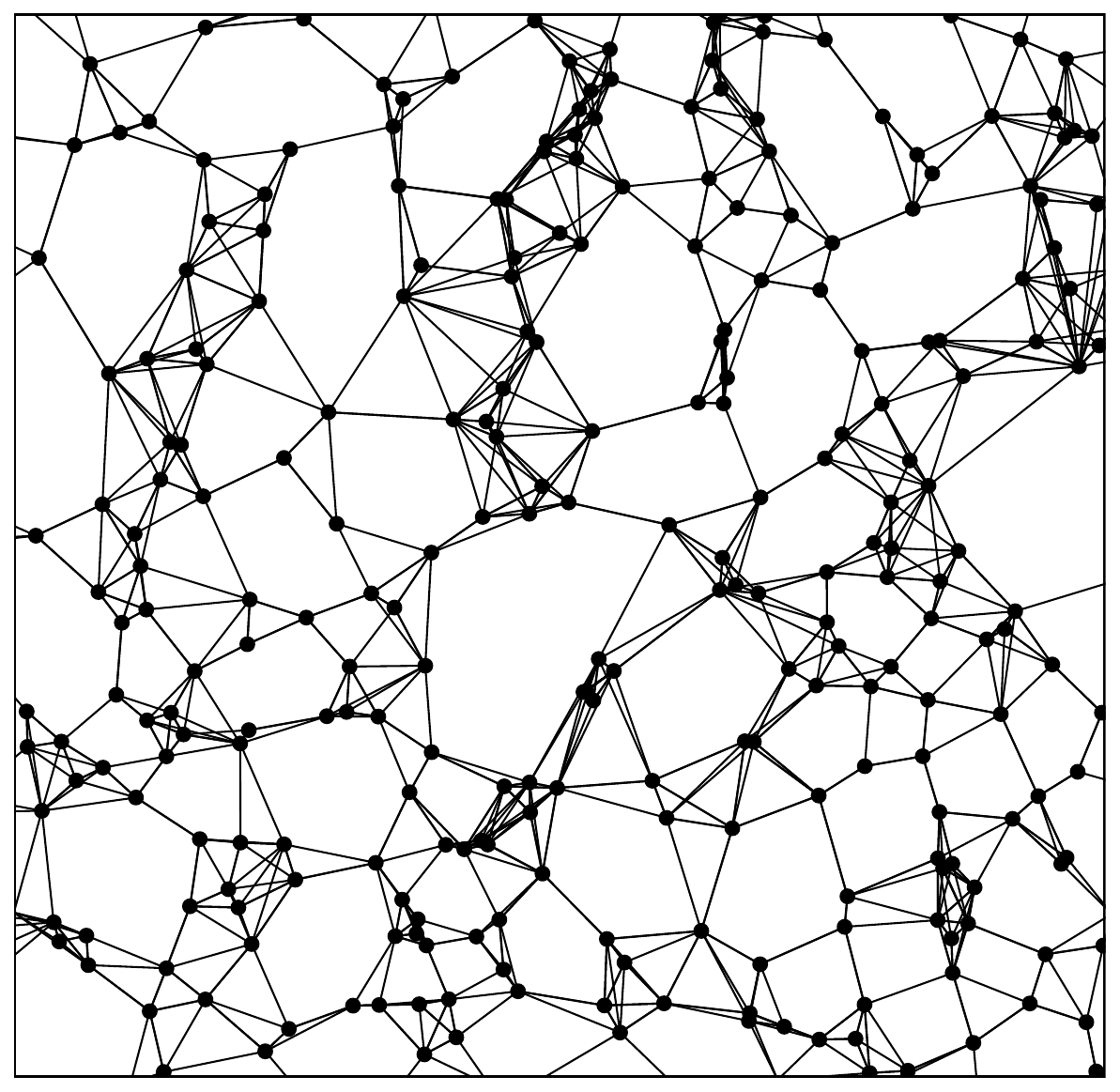}
    \caption{A realization of the NNE graph in $2$-dimensional Euclidean space.}
    \label{fig:SimulationEuclidean}
\end{figure}

Consider a stationary Poisson process $\eta$ in $\RR^d$. We sequentially construct a random graph with vertex set $\eta$ by connecting each point $x\in\eta$ to its nearest neighbour, then to its second nearest neighbour and so on. This is continued until $x$ is contained in the convex hull of the points already connected to $x$, that is, until $x$ lies in the convex hull of its $k$ nearest neighbours, where $k$ is chosen to be minimal. The resulting random graph is called the \emph{nearest neighbour embracing graph}, in short $\mathrm{NNE}(\eta)$. It was introduced by S.\,N.\ Chiu and I.\ Molchanov in \cite{CM}, and has found applications in good illumination problems \cite{ABM,AHBMP}. We also refer to \cite{CCCW} for a construction algorithm and to Figure \ref{fig:SimulationEuclidean} for a simulation {in Euclidean space}.

In the present paper we deal with geometric functionals of the nearest neighbour embracing graph in $\RR^d$ and in the $d$-dimensional hyperbolic space $\HH^d$. Studying random graphs in hyperbolic space of constant negative curvature has several motivations. First, changing the space underlying a geometric random graph model allows to distinguish between its universal properties and those which are curvature-dependent. Moreover, many real-world networks display a hyperbolic structure, which is well reflected by geometric random  graphs as studied in \cite{BogunaEtAl,FvdHMS,KrioukovEtAl}. To deal with both set-ups simultaneously we denote the underlying space by $\XX^d$ with $\XX\in\{\RR,\HH\}$ and by $d(\,\cdot\,,\,\cdot\,)$ the standard Riemannian metric on $\XX^d$. Moreover, we write $\vol$ for the Riemannian volume on $\XX^d$. For $x\in\XX^d$ and $r>0$, the notation $B(x,r)=\{y\in\XX^d:d(x,y)\leq r\}$ stands for the closed ball of radius $r$ centred at $x$. If $\XX^d=\RR^d$, we also use the notation $B_r:=B(0,r)$ for the ball of radius $r$ centred at the origin and we also denote for $\XX^d=\HH^d$ by $B_r$ the ball of radius $r$ centred at an arbitrary but fixed point $p\in\HH^d$, which we also refer to as the origin. 

Now, let $\eta$ be a stationary Poisson process on $\XX^d$, that is, $\eta$ is a Poisson process on $\XX^d$ with intensity measure equal to the Riemannian volume. Based on $\eta$ we construct the nearest neighbour embracing graph $\mathrm{NNE}(\eta)$ in $\XX^d$ as described above. Furthermore, we denote the set of \textit{outgoing} neighbours of $x\in\eta$ by $N(x,\eta)$, i.e., the set of neighbours the point $x$ connects to in order to close its convex hull. Finally, the set of \emph{all} neighbours of $x$ will be denoted by $\tilde{N}(x,\eta)$. The first class of geometric functionals associated with $\mathrm{NNE}(\eta)$ we are interested in are of the form
\[
F_t^{(\alpha)} := \frac{1}{2} \sum_{x \in \eta|_{B_t}} \sum_{y \in \tilde{N}(x,\eta)} d(x,y)^\alpha
\]
for some exponent $\alpha \ge 0$ and $t \ge 2$. Here, $\eta|_{B_t}$ denotes the restriction of $\eta$ to $B_t$. Note that in $F_t^{(\alpha)}$, every edge between two points within $B_t$ is counted with weight $1$, while edges which only have one endpoint in $B_t$ are counted with weight $1/2$. 

Our first result is a quantitative central limit theorem for $F_t^{(\alpha)}$, where the speed of convergence is measured by the Wasserstein and the Kolmogorov distance. We recall that the Wasserstein distance between the laws of two (real-valued) random variables $X$ and $Y$ is defined by
\[
\dist_W(X,Y) = \sup\{|\EE h(X) - \EE h(Y)| \colon h \in \mathrm{Lip}(1)\},
\]
where $\mathrm{Lip}(1)$ is the set of all functions $h \colon \RR \to \RR$ which are Lipschitz continuous with constant at most $1$. The Kolmorogov distance between the laws of $X$ and $Y$ is given by the supremum of the distance of their distribution functions, i.\,e.,
\[
\dist_K(X,Y) = \sup_{x \in \RR}|\PP(X \le x) - \PP(Y \le x)|.
\]
We can now formulate our first main result.

\begin{theorem}\label{thm:main}
    Let $N$ be a standard Gaussian random variable and $\diamondsuit\in\{K,W\}$. Then, for any $\alpha \ge 0$ and any $t \ge 2$, we have
    \[
    \dist_\diamondsuit \bigg( \frac{F_t^{(\alpha)} - \EE F_t^{(\alpha)}}{\sqrt{\Var F_t^{(\alpha)}}}, N \bigg) \leq\frac{c}{\sqrt{\vol(B_t)}},
    \]
    where $c>0$ is a constant only depending on the choice of $\XX^d$ and on $\alpha$.
\end{theorem}

The proof of Theorem \ref{thm:main} is based on the Malliavin--Stein method for functionals of Poisson processes. More precisely, we will use a normal approximation bound for such Poisson functionals as developed in \cite{LPS} with an extension taken from \cite{TT}. We remark at this point that an essential step in the proof of Theorem \ref{thm:main} will consist in establishing a lower variance bound for the functionals $F_t^{(\alpha)}$. More precisely, we will shown in Lemma \ref{lem:VarBd} below that for any $t\geq 2$ one has that
\[
\Var F_t^{(\alpha)} \ge c \,\vol(B_t),
\]
where $c>0$ is a constant depending only on $\alpha$ and the choice of $\XX^d$.

\begin{remark}
    We only consider non-negative exponents $\alpha \geq 0$. It is reasonable to expect that similar central limit results can be achieved when considering negative exponents $\alpha>-\frac{d}{2}$, by combining the proof of our Theorem~\ref{thm:main} with the bounds developed for the nearest neighbour graph in \cite{TT}. The speed of convergence we expect is then of the order of $\vol(B_t)^{-1/2}$ for $\alpha>-d/4$ and $\vol(B_t)^{-1+1/p}$ for $-d/2<\alpha\leq-d/4$, where $p \in (1,2]$ is chosen such that $2p\alpha+d>0$.
\end{remark}

\begin{figure}[t]
    \centering
    \includegraphics[width=.6\textwidth]{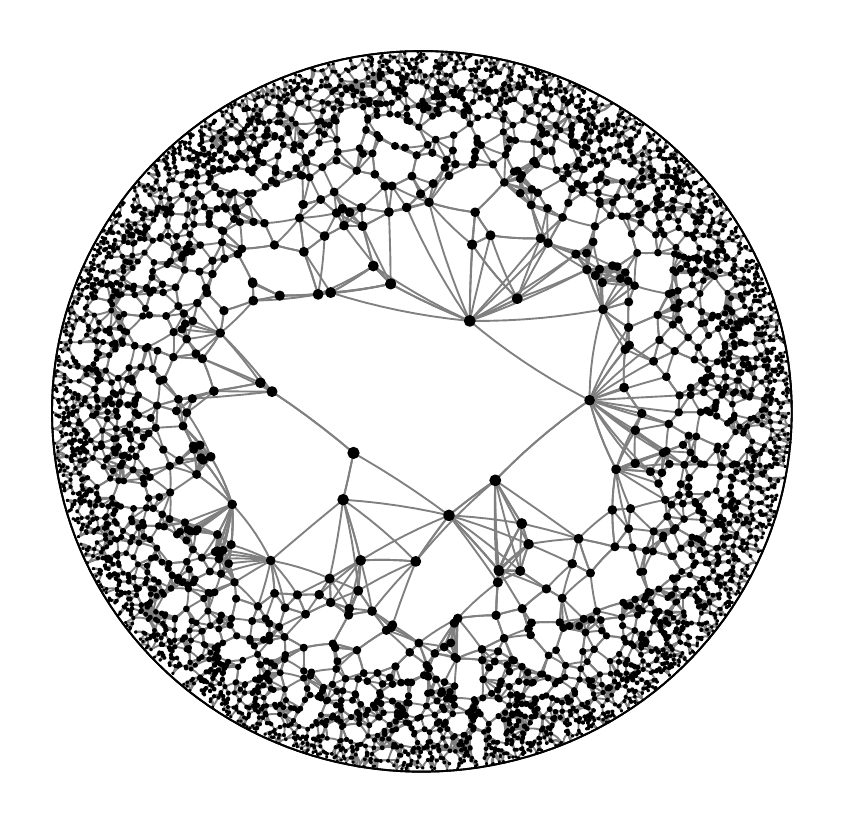}
    \caption{A realization of the NNE graph in $2$-dimensional hyperbolic space.}
    \label{fig:SimulationHyperbolic}
\end{figure}

\medspace

{Recalling that $N(x,\eta)$ is the set of neighbours} necessary to close the convex hull around $x$, we write
\[
\overline{\mathrm{deg}}(x) = \overline{\mathrm{deg}}(x,\eta) := |N(x,\eta)|.
\]
In \cite{CM}, the quantity $|N(x,\eta)|$ is referred to as the degree of $x$ in $\mathrm{NNE}(\eta)$, even if it strictly speaking differs from the actual degree of $x$ in the graph-theoretic sense (which equals $|\tilde{N}(x,\eta)|$). We will call $\overline{\mathrm{deg}}(x)$  the outdegree of $x$ in $\mathrm{NNE}(\eta)$. We denote the number of points of $\eta$ in $B_t$ with fixed outdegree $k$ by
\[
G_t^{(k)} := \sum_{x \in \eta|_{B_t}} \ind{\overline{\mathrm{deg}}(x)=k},
\]
where we assume $k \ge d+1$ since we almost surely need at least $d+1$ points to close the convex hull around $x$.

In \cite[Section 3]{CM}, a central limit theorem for $G_t^{(k)}$ was proven for $\XX^d = \RR^d$ equipped with the usual Euclidean metric under the assumption that the rescaled asymptotic variance is positive. Our second main contribution consists in a quantification of this result and in a removal of the additional assumption on the asymptotic variance.

\begin{theorem}\label{thm:main2}
    {Let $N$ be a standard Gaussian random variable and $\diamondsuit\in\{K,W\}$. Then, for any $k \ge d+1$ and any $t \ge 2$, we have
    \[
    \dist_\diamondsuit \bigg( \frac{G_t^{(k)} - \EE G_t^{(k)}}{\sqrt{\Var G_t^{(k)}}}, N \bigg) \leq\frac{c}{\sqrt{\vol(B_t)}},
    \]
    where $c>0$ is a constant only depending on the choice of $\XX^d$ and on $k$. Moreover,
    $$
    \dist_K \bigg( \frac{G_t^{(k)} - \EE G_t^{(k)}}{\sqrt{\Var G_t^{(k)}}}, N \bigg) \geq \frac{\hat{c}}{\sqrt{\vol(B_t)}} 
    $$
    for another constant $\hat{c}>0$ depending on the same parameters.
    }
\end{theorem}

As already anticipated above, an essential step in the proof of Theorem \ref{thm:main2} is a lower variance bound for $G_t^{(k)}$ of the form
\[
\Var G_t^{(k)} \ge c \,\vol(B_t),
\]
where $c>0$ is a constant depending only on $k$ and the choice of $\XX^d$. In the case of the Euclidean space, this shows that the variance assumption made in \cite{CM} is always fulfilled. {Furthermore, the proof of the upper bound in Theorem \ref{thm:main2} is again based on the Malliavin--Stein method for Poisson functionals in the form taken from \cite{LPS,TT}. However, we note that in contrast to Theorem \ref{thm:main} we also have here a lower bound on the rate of convergence if $\diamondsuit=K$. We expect the same to be true for $\diamondsuit=W$ and, similarly, for Theorem \ref{thm:main} as well.}

\medspace

The remaining parts of this paper are structured as follows. In Section \ref{sec:Preliminaries} we collect some background material about Poisson functionals and gather some geometric facts which will be used throughout the paper. In Section \ref{sec:RadiusStab} we consider the changes the nearest neighbour embracing graph undergoes if a new point is added to the generating point cloud. We study the local nature of these changes by analysing properties of the so-called radius of stabilization. The proofs of our main results are the content of the final Section \ref{sec:Proofs}.

\section{Preliminaries}\label{sec:Preliminaries}

\subsection{Poisson functionals}

Throughout this section we consider an arbitrary metric measure space, denoted by $(\YY,d,\lambda)$, and a Poisson process $\eta$ on $\YY$ with intensity measure $\lambda$. As a general reference to Poisson processes we refer to the monograph \cite{LP}. A random variable $F = f(\eta)$ only depending on $\eta$ is known as a Poisson functional. For $x\in\YY$ and a Poisson functional $F=f(\eta)$ we denote by $D_xF := f(\eta+\delta_x) - f(\eta)$ the difference operator, where $\eta+\delta_x$ is the configuration which arises if we add the point $x$ to $\eta$. Second order differences $D_{x,y} F$ are defined by iteration, i.\,e., $D_{x,y}F := D_y(D_xF) = f(\eta+\delta_x+\delta_y) - f(\eta+\delta_x) - f(\eta+\delta_y) + f(\eta)$ for $x,y \in \YY$.

The main tools we use in order to derive our central limit theorems are bounds established by Last, Peccati and Schulte in \cite{LPS}, with an extension taken from \cite{TT}. In particular, we use \cite[Theorem 6.1]{LPS} with $p_1 = p_2 = 1$ for the Wasserstein distance and derive an analogous bound from \cite[Theorem 3.4]{TT} with $p=2$ for the Kolmogorov distance. 

\begin{proposition}\label{thm:LPS6.1}
    Let $F$ be an integrable Poisson functional such that $\EE \int (D_x F)^2 \lambda (dx) < \infty$ and $\Var F > 0$, and let $N$ be a standard Gaussian random variable. Assume there are constants $c_1, c_2 > 0$ such that
    \begin{align}
        \mathbb{E}|D_x F|^5 &\le c_1\qquad \text{for $\lambda$-a.e. $x \in \YY$},\label{Cond1}\\
        \mathbb{E}|D_{x,y} F|^5 &\le c_2\qquad \text{for $\lambda^2$-a.e. $(x,y) \in \YY^2$}\label{Cond2}.
    \end{align}
    Writing $\bar{c} = \max\{1,c_1,c_2\}$, we have
    \begin{align*}
        d_W \Big( \frac{F - \EE F}{\sqrt{\Var F}}, N \Big) \le \, &\frac{3 \bar{c}}{\Var F} \Big[\int \Big( \int \PP(D_{x,y}F \ne 0)^{1/20} \lambda(dy) \Big)^2 \lambda(dx)\Big]^{1/2}\\
        &+ \frac{\bar{c}}{(\Var F)^{3/2}} \int \PP(D_x F \ne 0)^{2/5} \lambda (dx)
    \end{align*}
    as well as
    \begin{align*}
        d_K \Big( \frac{F - \EE F}{\sqrt{\Var F}}, N \Big) \le \,
        &\frac{3\bar{c}}{\Var F} \Big[\int \Big( \int \PP(D_{x,y}F \ne 0)^{1/20} \lambda(dy) \Big)^2 \lambda(dx)\Big]^{1/2}\\
        &+ \frac{2\bar{c}}{\Var F} \left(\int \PP(D_x F \ne 0)^{1/5} \lambda(dx)\right)^{1/2}\\
        &+\frac{13\bar{c}}{\Var F} \Big(\int \int \PP(D_{x,y}F \ne 0)^{1/20} \lambda(dy) \lambda(dx)\Big)^{1/2}.
    \end{align*}
\end{proposition}

{To derive a central limit theorem using Proposition \ref{thm:LPS6.1}, it is crucial to show a lower bound for the variance $\Var(F)$ of the right order. To do so in our set-up,} we shall use \cite[Theorem 1.1]{ST}, which reads as follows.

\begin{proposition}\label{thm:ST1.1}
    Let $F$ be an integrable Poisson functional such that $\EE \int (D_x F)^2 \lambda (dx) < \infty$ and
    \begin{equation}\label{CondST}
        \EE \Big[\int (D_{x,y} F)^2 \lambda^2(d(x,y))\Big] \le c \EE \Big[\int (D_x F)^2 \lambda(dx)\Big]
    \end{equation}
    for some constant $c \ge 0$. Then,
    \[
    \Var F \ge \frac{4}{(c+2)^2} \EE \Big[\int (D_x F)^2 \lambda(dx)\Big].
    \]
\end{proposition}

\subsection{Geometric ingredients}

In this section we gather the geometric ingredients needed for our purpose, including in particular basic definitions and facts from hyperbolic geometry. As a general references for hyperbolic geometry we mention the monographs \cite{CannonEtAl,Chavel,Ratcliffe}. The hyperbolic space $\HH^d$ is the unique simply connected $d$-dimensional Riemannian manifold of constant negative curvature $-1$. We denote the corresponding Riemannian metric by $d(\,\cdot\,,\,\cdot\,)$ and the Riemannian volume by $\vol$. There are several models of the hyperbolic space like the Beltrami--Klein model, the Poincar\'{e} ball model or the Poincar\'{e} half-plane model. We emphasize that all our results are model independent.

By $p$ we denote an arbitrarily chosen fixed point of $\HH^d$, referred to as the origin. Recalling the notation $B_r = B(p,r)$ for the closed ball of radius $r>0$ centred at $p$ it holds that
\begin{equation}\label{eq:volBt}
\vol (B_r) = {d\kappa_d}\int_0^r \sinh^{d-1}(u)\, du,
\end{equation}
where $\kappa_d := \vol(B(0,1))=\frac{\pi^{d/2}}{\Gamma(1+d/2)}$ is the volume of the $d$-dimensional Euclidean unit ball, see \cite[Eq.\ (3.26)]{Ratcliffe}. In fact, this identity is a direct consequence of the polar integration formula in hyperbolic geometry, which states that
\begin{equation}\label{eq:PolIntHypGeom}
    \int_{\HH^d} f(x) \,\vol(dx) = {d\kappa_d} \int_{S^{d-1}(p)} \int_0^\infty \sinh^{d-1}(u) f(\exp_p(uv))\,du\sigma_p(dv),
\end{equation}
see \cite[pages 123--125]{Chavel}. Here, $S^{d-1}(p)$ denotes the $(d-1)$-dimensional unit sphere in the tangent space $T_p\HH^d$ of $\HH^d$ at $p$, the measure $\sigma_p$ is the normalized spherical Lebesgue measure on $S^{d-1}(p) \subset T_p\HH^d$, and $\exp_p \colon T_p\HH^d \to \HH^d$ stands for the exponential map which in our case is applied to the point $uv \in T_p\HH^d$. {The point $\exp_p(uv)$ in $\HH^d$ arises from $p$ by traveling in $\HH^d$ distance $v>0$ along a geodesic ray in direction $u\in S^{d-1}(p)$.}

As can be seen from \eqref{eq:volBt}, the growth of $\vol(B_r)$ as a function of $r$ is exponential, which is in sharp contrast to the corresponding situation in $\RR^d$, where $\vol(B_r)=\kappa_dr^d$. In particular, it holds {in the hyperbolic setting} that
\begin{equation}\label{eq:hypvol}
    \gamma_d e^{r(d-1)}\le \vol(B_r)\leq\Gamma_d e^{r(d-1)},
\end{equation}
for all $r\geq 2$, where $\gamma_d,\Gamma_d>0$ are dimension-dependent constants, see \cite[Lemma 7]{HHT} or \cite[Lemma~4]{OttoThaele}.

\medspace

In our proofs it will often turn out to be possible to treat the $d$-dimensional hyperbolic space and the $d$-dimensional Euclidean space simultaneously. We shall use for this generic space the notation $\XX^d$ with $\XX\in\{\HH,\RR\}$. Moreover, we let $p\in\XX^d$ be a fixed origin, which for $\HH^d$ we choose as above and for $\RR^d$ we take for $p$ the point with Euclidean coordinates $(0,\ldots,0)$. For $z\in\XX^d$ we denote by $T_z\XX^d$ the tangent space of $\XX^d$ at $z$. We mention that in the Euclidean case, this is just $\RR^d$ itself. Similarly, we let $\exp_p:T_p\XX^d\to\XX^d$ be the exponential map at $p$, which in the hyperbolic case was introduced above. If $\XX^d=\RR^d$, the function $\exp_p$ is just the identity map on $\RR^d$.

\section{The radius of stabilization}\label{sec:RadiusStab}

A central element of our proofs is the fact that for the nearest neighbour embracing graph $\mathrm{NNE}(\eta)$, there exists a ``radius of stabilization'' for the difference operator $D_x$, i.\,e., a random variable $R(x,\eta)$ such that all changes between $\mathrm{NNE}(\eta)$ and $\mathrm{NNE}(\eta+\delta_x)$ happen within $B(x,R(x,\eta))$. The purpose of this section is to develop properties of this stabilization radius.

\begin{lemma}\label{lem:DOloc}
For $x \in \XX^d$, define 
\[
    R(x,\eta) := \max\{R_1(x,\eta),2R_2(x,\eta)\},
\]
where
\begin{align*}
  &R_1(x,\eta) := \max_{z \in N(x,\eta)} d(z,x),\\
  &R_2(x,\eta) := \max\{d(z,y) : z\in \eta, x\in N(z,\eta+\delta_x),y\in N(z,\eta)\}.
\end{align*}
Then for all $x \in \XX^d$, $\alpha \in \RR$ and $k \geq d+1$, it holds that
\begin{align*}
    &D_x F_t^{(\alpha)}(\eta) = D_xF_t^{(\alpha)} (\eta \cap B(x,R(x,\eta)))\\
    \intertext{and}
    &D_x G_t^{(k)}(\eta) = D_xG_t^{(k)} (\eta \cap B(x,R(x,\eta))).
\end{align*}
\end{lemma}
\begin{proof}
    If $x$ is added to the configuration, any new edge must have $x$ as one of its endpoints. Hence, there are two possibilities for new edges:
    \begin{itemize}
        \item edges from $x$ to one of its nearest neighbours, whose lengths are smaller than $R_1(x,\eta)$ by definition,
        \item edges from $x$ to some $z \in \eta$ such that $x \in N(z,\eta+\delta_x)$. In particular, $d(x,z) \le \sup_{y \in N(z,\eta)} d(z,y)$ and hence, $z \in B(x,R_2(x,\eta))$.
    \end{itemize}
    On the other hand, edges which are deleted upon the addition of $x$ are edges from $z \in \eta$ to $y \in N(z,\eta)$ such that $x \in N(z,\eta + \delta_x)$ and $\conv(x, N(z,\eta) \setminus \{y\}) \subset \conv(N(z,\eta))$, i.\,e., $x$ completes the convex hull around $z$ before $y$ can get added to $N(z,\eta+\delta_x)$. This implies necessarily that $d(x,z) \le \sup_{w \in N(z,\eta)} d(w,z)$ and also $d(x,y) \le 2 \sup_{w \in N(z,\eta)} d(w,z)$, so that $z,y \in B(x,2R_2(x,\eta))$. Hence, all changes to the edges or node degree of the graph (and, by extension, to $F_t^{(\alpha)}$ or $G_t^{(k)}$) caused by insertion of the point $x$ happen within $B(x,R(x,\eta))$.
\end{proof}

To derive properties of the radius of stabilization $R(x,\eta)$ of Lemma \ref{lem:DOloc} we start with the following technical result. Recall that we denote by $S^{d-1}(z)\subset T_z\XX^d$ the unit sphere in the tangent space at $z$. For $v\in S^{d-1}(z)$ we let $H_v(z)$ be the halfspace in $T_z\XX^d$ containing $v$, whose bounding hyperplane passes through $z$ and which has normal vector $v-z$. Further, for $z\in\XX^d$, $v\in S^{d-1}(z)$ and $r>0$ we put
\begin{equation}\label{eq:S(z,v,r)}
S(z,v,r) := B(z,r) \cap \exp_z(H_v(z)),
\end{equation}
where we recall that $\exp_z:T_z\XX^d\to\XX^d$ stands for the exponential map. We remark that in the Euclidean case $\XX^d=\RR^d$ the set $S(z,v,r)$ is just the half-sphere centred at $z$ with radius $r$ pointing in the direction of $v$. Moreover, in order to treat the Euclidean and hyperbolic cases in parallel, we introduce the function $\rho:(0,\infty) \rightarrow (0,\infty)$, defined by
\[
\rho(r) := 
\begin{cases}
    \kappa_d r^d & \text{in the Euclidean case}\\
    \gamma_d e^{r(d-1)}\ind{r\geq2} & \text{in the hyperbolic case.}
\end{cases}
\]
Here, {we recall that
$\kappa_d$ denotes} the volume of the $d$-dimensional Euclidean unit ball, and $\gamma_d$ is the constant from \eqref{eq:hypvol}. In other words, $\rho(r)$ equals the volume of a ball of radius $r$ in the Euclidean case, while in the hyperbolic case, it is a lower bound on the volume of a ball of radius $r$ which {as $r\to\infty$} is of the same order as $\vol(B(p,r))$ by \eqref{eq:hypvol}.

\begin{lemma}\label{lem:exptails}
    For any $r > 0$ and $z \in \XX^d$, there exist constants $c_1, c_2 \in (0,\infty)$ depending on $d$ only such that
    \[
    \PP(\exists\, v \in S^{d-1}(z) \text{ such that } \eta(S(z,v,r)) = 0) \le c_1 \exp(-c_2 \rho(r)).
    \]
\end{lemma}

\begin{proof}
    By stationarity of $\eta$ we may assume that {$z$ equals the origin $p$.}
    We cover the tangent space $T_p\XX^d$, which we can identify with $\RR^d$ in both cases, with cones $C_1,\ldots,C_K$ centered at $p$ having angular radius $\pi/4$. Now, let $r>0$ and $v \in S^{d-1}(p)$ be such that $\eta(S(p,v,r))=0$.
    Then $v \in C_i \subset H_v(p)$ for some $i \in \{1,\ldots,K\}$ and thus $D_i:=\exp_p(C_i) \subset \exp_p(H_v(p))$. Hence, $D_i \cap B(p,r) \subset S(p,v,r)$, so that $\eta(D_i \cap B(p,r)) = 0$. Application of the union bound leads to
    \begin{align*}
        \PP(\exists\, v \in S^{d-1}(p) \text{ such that } \eta(S(p,v,r))=0) & \leq \sum_{i=1}^K\PP(\eta(D_i\cap B(p,r))=0)\\
        &= \sum_{i=1}^K \exp(-\vol(D_i \cap B(p,r))).
    \end{align*}
    Since each set of the form $D_i$, with $i\in\{1,\ldots,K\}$, can be transformed into $D_j$ for each $j\in\{1,\ldots,K\}$ using an isometry of the space $\XX^d$, $\vol(D_1 \cap B(p,r))=\ldots=\vol(D_K \cap B(p,r))=A\vol(B(p,r))$ for some constant $A>0$. Using the definition of the function $\rho(r)$ and putting $c_1:=K$ and $c_2:=A$, this leads to the desired result.
\end{proof}

The previous result allows us to derive tail estimates for the radius of stabilization. They will turn out to be of crucial importance in later parts of the proof of Theorem \ref{thm:main}.

\begin{lemma}\label{lem:exptails3}
    For $m \in \NN_0$, let $x,y_1,\hdots,y_m\in \XX^d$ be distinct points. Then there are constants $c_0,C_0>0$ depending on the choice of $\XX^d$ only such that
    \[
    \PP(R(x,\eta + \delta_{y_1} + \hdots + \delta_{y_m})>s) \leq (m+1) C_0 \exp(-c_0\rho(s/2)),\qquad s>0.
    \]
\end{lemma}
\begin{proof}
    We start with the case $m=0$. Recall the definitions of $R(x,\eta),\ R_1(x,\eta)$ and $R_2(x,\eta)$ from Lemma~\ref{lem:DOloc}. First, we establish estimates for $R_1(x,\eta)$. Note that the event $\{R_1(x,\eta)>s\}$ implies that $x$ is not inside the convex hull of its neighbours within $B(x,s)$, which means there must be {a set of the form \eqref{eq:S(z,v,r)}} which does not contain any points of $\eta$. In other words,
    \begin{align}
        \PP(R_1(x,\eta)>s) &\leq \PP\left(\exists\, v \in S^{d-1}(x) \text{ such that } \eta(S(x,v,s))=0 \right)\notag\\
        &\leq c_1 \exp(-c_2 \rho(s)), \label{eq:P1.4}
    \end{align}
    where the second bound follows from Lemma~\ref{lem:exptails}.
    
    We now deal with $R_2(x,\eta)$. We have
    \begin{align*}
        \PP(R_2(x,\eta)>s) &= \PP\left(\max\{d(z,y) : z\in {\eta}, x\in N(z,{\eta}+\delta_x),y\in N(z,{\eta})\}>s\right)\\
        &\leq \EE \int_{\XX^d}\int_{\XX^d} \ind{d(z,y)>s} \ind{x \in N(z,\eta+\delta_x)} \ind{y \in N(z,\eta)} \,\eta(dz)\eta(dy).
    \end{align*}
    Applying the Mecke equation for Poisson processes \cite[Theorem~4.4]{LP} and the Cauchy--Schwarz inequality, we deduce that
    \begin{multline*}
    \PP(R_2(x,\eta)>s) \\\leq \int_{\XX^d} \int_{\XX^d} \ind{d(z,y)>s} \PP(x\in N(z,\eta + \delta_x + \delta_y))^{1/2} \PP(y \in N(z, \eta+\delta_x+\delta_y))^{1/2}\,\vol(dz)\vol(dy).
    \end{multline*}
    Note that adding points to $\eta$ decreases the probability that $x$ is a neighbour of $z$, therefore
    \[
    \PP(x\in N(z,\eta + \delta_x + \delta_y)) \leq \PP(x\in N(z,\eta + \delta_x)),
    \]
    and likewise for $y$. Moreover, $x$ can only be among the convex neighbours of $z$ if there is a half-ball of radius $d(x,z)$ centred at $z$ which does not contain any points of $\eta$. We thus have
    \begin{align}
        \PP(R_2(x,\eta)>s) &\leq \int_{\XX^d} \int_{\XX^d} \ind{d(z,y)>s} \PP(x\in N(z,\eta + \delta_x))^{1/2} \PP(y \in N(z, \eta+\delta_y))^{1/2} \,\vol(dz)\vol(dy)\notag\\
        &\leq \int_{\XX^d} \int_{\XX^d}\ind{d(z,y)>s} \PP(\exists\, v \in S^{d-1}(z) \text{ such that } \eta(S(z,v,d(x,z)))=0)^{1/2} \notag\\
        &\hspace{1cm}\times\PP(\exists\, w \in S^{d-1}(z) \text{ such that } \eta(S(z,w,d(y,z)))=0)^{1/2} \,\vol(dz)\vol(dy). \label{eq:P1.1}
    \end{align}
    By Lemma~\ref{lem:exptails} we deduce that \eqref{eq:P1.1} is bounded by
    \begin{equation}\label{eq:P1.2}
        \int_{\XX^d} \int_{\XX^d} c_1^2 \ind{d(z,y)>s} \exp(-c_2\rho(d(z,x))) \exp(-c_2\rho(d(z,y))) \,\vol(dz)\vol(dy).
    \end{equation}
    Since distance and volume are invariant under isometries, in both the Euclidean and the hyperbolic case we may assume that the point $x$ coincides with the origin. Then, in the Euclidean case, \eqref{eq:P1.2} is equal to
    \[
    \int_{\RR^d} \int_{\RR^d} c_1^2 \ind{\lVert z-y\rVert>s} \exp(-c_2\kappa_d \lVert z\rVert^d) \exp(-c_2\kappa_d \lVert z-y\rVert^d) \,\vol(dz)\vol( dy).
    \]
    Substituting $y=z+v$, this is 
    \[
    c_1^2 \int_{\RR^d}  \ind{\lVert v\rVert>s}  \exp(-c_2\kappa_d \lVert v\rVert^d) \,\vol(dv) \int_{\RR^d} \exp(-c_2\kappa_d \lVert z\rVert^d)\,\vol(dz) .
    \]
    Transformation into spherical coordinates and the subsequent  change of variables $c_2\kappa_d r^d=u$ yields that this is the same as
    \begin{align*}
        &c_1^2(d\kappa_d)^2\int_s^\infty e^{-c_2\kappa_d r^d} {r^{d-1}}\,dr\,\int_0^\infty e^{-c_2\kappa_d r^d} {r^{d-1}}\,dr=\frac{c_1^2}{c_2^2}\int_{c_2\kappa_ds^d}^\infty e^{-u}\,du\,\int_0^\infty e^{-u}\,du=\frac{c_1^2}{c_2^2}e^{-c_2\kappa_ds^d},
    \end{align*}
    and we conclude that
    \begin{equation}\label{eq:P1.5}
        \PP(R_2(x,\eta)>s) \leq \frac{c_1^2}{c_2^2} \exp(-c_2 \kappa_d s^d) = \frac{c_1^2}{c_2^2} \exp(-c_2 \rho(s)).
    \end{equation}

    In the hyperbolic case, the right hand side of \eqref{eq:P1.2} is equal to
    \[
    \int_{\HH^d} \int_{\HH^d} c_1^2 \ind{d(y,z)>s} \exp\left(-c_2(\rho(d(p,z))\right) \exp(-c_2\rho(d(y,z)))\,\vol( dz)\vol( dy).
    \]
    Writing again $\exp_z:T_z\mathbb{H}^d\to\mathbb{H}^d$ for the exponential map at $z$ and substituting $y=\exp_z(ru)$ with $r>0$ and $u\in S^{d-1}(z)\subset T_z\HH^d$. Then the last expression can be rewritten as
    \begin{align*}
d\kappa_d\int_{\mathbb{H}^d}\int_{S^{d-1}(z)}\int_s^\infty c_1^2 (\sinh r)^{d-1} \exp(-c_2\rho(d(p,z))) \exp(-c_2\rho(r))\,dr\sigma_z(du)\vol(dz),
    \end{align*}
    where $\sigma_z$ is the normalized spherical Lebesgue measure on $S^{d-1}(z)$.
    Next, we notice that this expression does not depend on the direction $u$. Together with the introduction of hyperbolic spherical coordinates as in \eqref{eq:PolIntHypGeom}, this allows us to reduce the previous expression to 
    \begin{align*}
        c_1^2(d\kappa_d)^2\int_0^\infty (\sinh r)^{d-1}e^{-c_2\rho(r)}\,dr\int_{s}^\infty (\sinh r)^{d-1}e^{-c_2\rho(r)}\,dr.
    \end{align*}
    To derive an upper bound for the last expression we use that $\sinh x=(e^x-e^{-x})/2\leq e^{x}/2\leq e^x$ for $x>0$. Then, using the substitution $u=c_2\gamma_de^{r(d-1)}$ and assuming that $s\geq2$, we arrive at
    \begin{align*}
        &c_1^2(d\kappa_d)^2\Big(\int_0^2 e^{r(d-1)}\,dr+\int_2^\infty e^{r(d-1)}e^{-c_2\gamma_d e^{r(d-1)}}\,dr\Big)\Big(\int_s^\infty e^{r(d-1)}e^{-c_2\gamma_d e^{r(d-1)}}\,dr\Big)\\
        &=c_1^2(d\kappa)_d^2\Big(c_3+\frac{1}{(d-1)\gamma_dc_2}\int_{0}^\infty e^{-u}\,du\Big)\Big(\frac{1}{(d-1)c_2\gamma_d}\int_{c_2\gamma_d e^{s(d-1)}}^\infty e^{-u}\,du\Big)\\
        &=A_1 e^{-c_2\gamma_d e^{s(d-1)}}\\
        &=A_1 e^{-c_2\rho(s)},
    \end{align*}
    where $c_3,A_1>0$ are constants only depending on $\HH^d$. Altogether we derive the bound 
    \begin{equation}\label{eq:P1.6}
    \PP(R_2(x,\eta)>s) \leq A_1 e^{-c_2\rho(s)}
    \end{equation}
    in the case that $s \geq 2$. Note that if $0<s<2$, the bound continues to hold trivially, provided that we choose $A_1\geq1$ {large enough}. Summarizing, this shows that the statement of the lemma holds in the case $m=0$ in both the Euclidean and the hyperbolic settings.
    
    Let us now study $R(x,\mu+\delta_y)$, where $x,y \in \XX^d$ are distinct points and $\mu \subset \XX^d$ is a locally finite collection of points. We have
    \[
    R(x,\mu+\delta_y) = \max\{R_1(x,\mu+\delta_y), 2 R_2(x,\mu+\delta_y)\}.
    \]
    Note that the addition of a point $y$ can only make the convex hull of neighbours of $x$ shrink, hence
    \[
    R_1(x,\mu+\delta_y) = \max\{d(x,z) : z \in N(x,\mu+\delta_y)\} \leq \max\{d(x,z) : z \in N(x,\mu)\} = R_1(x,\mu).
    \]
    As for $R_2(x,\mu+\delta_y)$, we have
    \[
    R_2(x,\mu+\delta_y) = \max\{d(z,w) : z \in \mu+\delta_y, x \in N(z,\mu+\delta_x+\delta_y), w\in N(z,\mu+\delta_y)\}.
    \]
    We first treat the case $z=y$. If $x \in N(y,\mu+\delta_x+\delta_y)$, then the maximum reduces to
    \[
    \max\{d(y,w) : w \in N(y,\mu + \delta_y)\} \leq R_1(y,\mu).
    \]
    We are left to find a bound for
    \begin{equation}\label{eq:P1.3}
        \max\{d(z,w) : z \in \mu, x \in N(z,\mu+\delta_x+\delta_y), w\in N(z,\mu+\delta_y)\}.
    \end{equation}
    If $y \in N(z,\mu + \delta_y)$, then necessarily $d(y,z) \leq \max\{d(z,w):w \in N(z,\mu)\}$, and hence \eqref{eq:P1.3} is bounded by
    \begin{equation}\label{eq:P1.7}
        \max\{d(z,w) : z \in \mu, x \in N(z,\mu+\delta_x+\delta_y), w\in N(z,\mu)\}.
    \end{equation}
    Lastly, we observe that $x \in N(z,\mu+\delta_x+\delta_y)$ implies that $x \in N(z,\mu+\delta_x)$, since adding points can only shrink the convex hull. Thus \eqref{eq:P1.7} is smaller than
    \[
    \max\{d(z,w) : z \in \mu, x \in N(z,\mu+\delta_x), w\in N(z,\mu)\} = R_2(x,\mu).
    \]
    Combining the above estimates, we get
    \[
    R(x,\mu+\delta_y) \leq \max\{R_1(x,\mu), 2\max\{R_1(y,\mu) , R_2(x,\mu)\}) = \max\{R(x,\mu) , {2} R_1(y,\mu)\}.
    \]
    By induction, it now follows that
    \[
    R(x,\mu+\delta_{y_1} + \hdots + \delta_{y_m}) \leq \max\{R_1(x,\mu), 2R_2(x,\mu) , 2R_1(y_1,\mu) , \hdots , 2R_1(y_m,\mu)\}.
    \]
    Hence,
    \begin{multline*}
        \PP(R(x,\eta+\delta_{y_1} + \hdots + \delta_{y_m})>s) \\\leq \PP(R_1(x,\eta) > s) + \PP(R_2(x,\eta)>s/2) + \PP(R_1(y_1,\eta)>s/2) + \hdots + \PP(R_1(y_m,\eta)>s/2).
    \end{multline*}
    Using \eqref{eq:P1.4}, \eqref{eq:P1.5} and \eqref{eq:P1.6}, we deduce that 
    \[
    \PP(R(x,\eta+\delta_{y_1} + \hdots + \delta_{y_m})>s) \leq C_0 (m+1) \exp\left( - c_0 \rho(s/2)\right),
    \]
    for well-chosen constants $C_0,c_0>0$ depending on the choice of $\XX^d$ only.
\end{proof}

In the next lemma, we prove an upper bound on moments of the number of points of $\eta$ inside a ball of radius $R(x,\eta)$.

\begin{lemma}\label{lem:etamom}
    Let $\cA \subset \XX^d$ be a finite set of points and let $x \in \XX^d$. For any $m \in \NN$, there is a constant $c>0$ only depending on $m$, $\cA$ and the choice of $\XX^d$ such that
    \[
    \EE \eta(B(x,R(x,\eta \cup \cA))^m \leq c.
    \]
\end{lemma}
\begin{proof}
    For any natural number $m\in \NN$, we have
    \begin{align*}
        &\EE \eta(B(x,R(x,\eta \cup \cA)))^m\\
        =\,& \EE \Big(\int_{\XX^d} \ind{d(x,y) \le R(x,\eta \cup \cA)} \,\eta(dy)\Big)^m\\
        = \, &\sum_{\ell=1}^m p_\ell(m) \EE \sum_{(y_1,\ldots,y_\ell) \in \eta_{\ne}^\ell} \ind{d(x,y_1) \le R(x,\eta \cup \cA)} \cdot \ldots \cdot \ind{d(x,y_\ell) \le R(x,\eta \cup \cA)}\\
        = \, &\sum_{\ell=1}^m p_\ell(m) \int_{(\XX^d)^\ell} \EE \Big[\ind{d(x,y_1) \le R(x,\eta+\delta_{y_1}+ \ldots + \delta_{y_\ell} \cup \cA)}\\
        &\hspace{2cm}\cdot \ldots \cdot \ind{d(x,y_\ell) \le R(x,\eta+\delta_{y_1}+ \ldots + \delta_{y_\ell} \cup \cA)}\Big]\, \vol(dy_1)\ldots \vol(dy_m)\\
        \le \, &\sum_{\ell=1}^m p_\ell(m) \int_{(\XX^d)^\ell} \PP(d(x,y_1) \le R(x,\eta+\delta_{y_1}+ \ldots + \delta_{y_\ell} \cup \cA))^{1/\ell}\\
        &\hspace{2cm}\cdot \ldots \cdot \PP(d(x,y_\ell) \le R(x,\eta+\delta_{y_1}+ \ldots + \delta_{y_\ell} \cup \cA))^{1/\ell}\, \vol(dy_1) \ldots \vol(dy_\ell),
    \end{align*}
    where $p_\ell(m)$ is the number of possibilities to partition $\{1,\ldots,m\}$ into $\ell$ non-empty subsets, see \cite[Section 5.2]{Bona}, and $\eta_{\neq}^\ell$ denotes the set of $\ell$-tuples of distinct points of $\eta$. We have used the Mecke equation for Poisson processes of \cite[Theorem~4.4]{LP} to achieve the third line and Hölder's inequality for the last line. By Lemma~\ref{lem:exptails3}, we derive that
    \[
    \EE \eta(B(x,R(x,\eta \cup \cA)))^m \leq \sum_{\ell=1}^m p_\ell(m) \left( \int_{\XX^d} C_0 (\ell + |\cA| + 1) \exp\left(-\frac{c_0}{\ell}\rho(d(x,y))\right) \,\vol(dy)\right)^\ell \leq c
    \]
    for a suitable constant $c>0$ only depending on $m$, $\cA$ and the choice of $\XX^d$, using similar arguments as in the proof of Lemma~\ref{lem:exptails3}.
\end{proof}

\section{Proofs of the main results}\label{sec:Proofs}

\subsection{Proof of Theorem \ref{thm:main}}

The proof of Theorem~\ref{thm:main} requires estimates for the add-one costs as well as a lower variance bound. We start with the add-one cost bounds and introduce for $x\in\XX^d$ and a set $A\subset\XX^d$ the notation $d(x,A):=\inf\{d(x,a):a\in A\}$ for the distance of $x$ to $A$. Moreover, we shall write $\partial B_r=\{x\in\XX^d:d(x,p)=r\}$ for the sphere in $\XX^d$ of radius $r>0$ centred at $p$. 

\begin{lemma}\label{lem:DxBounds}
    Let $\alpha\geq 0$ and $\beta>0$. Then there is a constant $\tilde{C}_1>0$ depending on $\beta,\alpha$ and the choice of $\XX^d$ such that for all $x,y \in \XX^d$ and $t\geq 2$,
    \begin{equation}\label{eq:DxBounds}
        \EE \left| D_x F_t^{(\alpha)} \right|^\beta \leq \tilde{C}_1 \quad \text{and} \quad \EE \left| D_{x,y} F_t^{(\alpha)} \right|^\beta \leq \tilde{C}_1.
    \end{equation}
    Moreover, there are constants $\tilde{c}_2,\tilde{C}_2>0$ depending on the choice of $\XX^d$ only such that for all $x,y \in \XX^d$ and $t\geq 2$,
    \[
    \PP\left( D_{x,y} F_t^{(\alpha)} \neq 0 \right) \leq \tilde{C}_2 \exp(-\tilde{c}_2 \rho(d(x,y)/2)),
    \]
    and, if $x \notin B_t$,
    \[
    \PP\left( D_x F_t^{(\alpha)} \neq 0 \right) \leq \tilde{C}_2 \exp(-\tilde{c}_2 \rho(d(x,\partial B_t)/2))
    \]
    and
    \[
    \PP\left( D_{x,y} F_t^{(\alpha)} \neq 0 \right) \leq \tilde{C}_2 \exp(-\tilde{c}_2 \rho(d(x,y)))\exp(-\tilde{c}_2 \rho(d(x,\partial B_t))).
    \]
\end{lemma}

\begin{proof}
    Let us first check \eqref{eq:DxBounds}. If a point $x$ is added, it follows from Lemma \ref{lem:DOloc} that the number of edges which are newly gained is bounded by $\eta(B(x,R(x,\eta)))$. On the other hand, a rough bound on the number of edges which are lost is given by the total number of possible edges within $B(x,R(\eta))$, which is bounded by $\eta(B(x,R(\eta)))^2$. Altogether, this yields
    \[
    |D_x F_t^{(\alpha)}| \le 2\eta(B(x,R(x,\eta)))^2 R(x,\eta)^\alpha,
    \]
    and hence, for any $\beta>0$,
    \[
    \EE|D_x F_t^{(\alpha)}|^\beta \le 2^\beta\EE[\eta(B(x,R(x,\eta)))^{4\beta}]^{1/2}\EE[R(x,\eta)^{2\alpha \beta}]^{1/2}.
    \]
    We have that for $m = \lceil 4\beta \rceil$,
    \[
    \EE[\eta(B(x,R(x,\eta)))^{4\beta}] \leq \EE[\eta(B(x,R(x,\eta)))^m]^{4\beta/m},
    \]
    which, by Lemma~\ref{lem:etamom}, is bounded by a constant depending on $\beta$ and the choice of $\XX^d$ {only}. Furthermore, the term $\EE[R(x,\eta)^{2\alpha \beta}]$ is bounded as well by a constant depending on $\alpha,\beta$ and the choice of $\XX^d$ by Lemma~\ref{lem:exptails3}.

    By similar arguments as {we used for $D_xF_t^{(\alpha)}$}, we have
    \[
    |D_{x,y} F_t^{(\alpha)}| \le 2\eta(B(x,R(x,\eta)))^2  R(x,\eta)^\alpha  + 2(\eta+\delta_y)(B(x,R(x,\eta+\delta_y)))^2 R(x,\eta+\delta_y)^\alpha,
    \]
    and hence, for any $\beta>0$,
    \begin{align*}
        \EE|D_{x,y} F_t^{(\alpha)}|^\beta \le \, &2^{2\beta} \bigg(\EE[\eta(B(x,R(x,\eta)))^{4\beta}]^{1/2} \cdot \EE[ R(x,\eta)^{2\alpha \beta}]^{1/2} \bigg) \\
        &+2^{2\beta} \bigg( \EE[\big(\eta(B(x,R(x,\eta+\delta_y)))+1\big)^{4\beta}]^{1/2} \cdot \EE[ R(y,\eta+\delta_x)^{2\alpha \beta}]^{1/2} \bigg).
    \end{align*}
    This is again bounded by a constant depending on $\alpha,\beta$ and the choice of $\XX^d$.

    To find a bound on $\PP(D_{x,y} F_t^{(\alpha)} \ne 0)$, recall that as argued in \cite{LPS} (cf.\ the discussion below Proposition 1.4 there, for instance), if $d(x,y) > R(x,\eta)$ and $R(x,\eta+\delta_y) = R(x,\eta)$, we have $D_{x,y} F = 0$ for any Poisson functional $F$ with radius of stabilization $R(x,\eta)$. In view of the definition of $R(x,\eta)$ from Lemma \ref{lem:DOloc}, we see that $R(x,\eta+\delta_y) \ne R(x,\eta)$ is possible if any of the following situations occur:
    \begin{itemize}
        \item $y \in N(x,\eta+\delta_y)$, in which case we must have $d(x,y) \le \sup_{z \in N(x,\eta)} d(x,z) \le R(x,\eta)$,
        \item there exists $z$ such that $x \in N(z, \eta + \delta_x)$, $y \in N(z, \eta+\delta_y)$, which implies that $d(y,z) \le \sup_{w \in N(z,\eta)} d(z,w)$ and hence $d(x,y) \le d(x,z) + d(y,z) \le R(x,\eta)$,
        \item $x \in N(y,\eta+\delta_x)$, in which case $d(x,y) \le \sup_{z \in N(y,z)} d(y,z) \le R(y,\eta)$.
    \end{itemize}
    Altogether, this leads to
    \begin{align}
        \PP(D_{x,y} F_t^{(\alpha)} \ne 0) &\le 3\PP(R(x,\eta) \ge d(x,y)) + \PP(R(y,\eta) \ge d(x,y))\notag\\
        &\le 4C_0 \exp(-c_0\rho(d(x,y)/2))
    \end{align}
    using Lemma \ref{lem:exptails3}.

    Now note that if $x \notin B_t$ and $d(x,\partial B_t) > R(x,\eta)$, then any edge which is affected by adding $x$ has both endpoints outside $B_t$, and consequently, $D_x F_t^{(\alpha)} = 0$. This yields
    \[
    \PP(D_x F_t^{(\alpha)} \ne 0) \le \PP(R(x,\eta) \ge d(x, \partial B_t)) \le C_0 \exp(-c_0 \rho(d(x,\partial B_t)/2)),
    \]
    again by Lemma \ref{lem:exptails3}.

    To show the last bound, we need to combine both of the previous arguments. Assume that $x \notin B_t$. First of all, note that, as above, if $d(x,y)>R(x,\eta)$ and $R(x,\eta+\delta_y) = R(x,\eta)$ then $D_{x,y} F_t^{(\alpha)} = 0$. Moreover, if
    \[
    d(x,\partial B_t) > \max \{R(x,\eta+\delta_y),R(x,\eta)\},
    \]
    then both $D_xF_t^{(\alpha)}(\eta)=0$ and $D_xF_t^{(\alpha)}(\eta+\delta_y)=0$ and consequently $D_{x,y} F_t^{(\alpha)} = 0$. Hence, reusing arguments from the beginning of the proof, we deduce that $D_{x,y} F_t^{(\alpha)} \ne 0$ implies $d(x,y)<\max\{R(x,\eta),R(y,\eta)\}$ and $d(x,\partial B_t)<\max\{R(x,\eta),R(x,\eta+\delta_y)\}$. It follows that
    \begin{align*}
        \PP(D_{x,y} F_t^{(\alpha)} \ne 0) &\leq \PP(R(x,\eta)>\max\{d(x,y),d(x,\partial B_t)\}) \\
        &\phantom{\le}+ \PP(R(y,\eta)>\max\{d(x,y),d(x,\partial B_t)\}) \\
        &\phantom{\le}+ \PP(R(x,\eta+\delta_y)>\max\{d(x,y),d(x,\partial B_t)\})\\
        &\le 3 C_0 \exp(-c_0 \rho(\max\{d(x,y),d(x,\partial B_t)\}/2)).
    \end{align*}
    Note that $\rho(\max\{a,b\}) = \max\{\rho(a),\rho(b)\} \ge \frac{1}{2} (\rho(a)+\rho(b))$, since $\rho$ is a non-decreasing function. We thus derive that
    \[
    \PP(D_{x,y} F_t^{(\alpha)} \ne 0) \leq 3 C_0 \exp\left(- \frac{c_0}{2} \rho(d(x,y)/2)\right) \exp\left(- \frac{c_0}{2} \rho(d(x,\partial B_t)/2)\right).
    \]
    This concludes the proof.
\end{proof}

In a next step we derive a lower variance bound for $F_t^{(\alpha)}$ based on Proposition \ref{thm:ST1.1}.

\begin{lemma}\label{lem:VarBd}
    For $t \geq 2$ it holds that
    \[
    \Var F_t^{(\alpha)} \ge c \rho(t),
    \]
    where $c>0$ is a constant depending on $\alpha$ and the choice of $\XX^d$ only.
\end{lemma}
\begin{proof}
    We use Proposition \ref{thm:ST1.1}. First note that
    \begin{align*}
        \int_{\XX^d\times\XX^d} \EE (D_{x,y} F_t^{(\alpha)})^2 \,\vol(dx)\vol(dy)
        &= \int_{\XX^d\times\XX^d} \EE [\mathbbm{1}\{D_{x,y} F_t^{(\alpha)} \ne 0 \} (D_{x,y} F_t^{(\alpha)})^2] \,\vol(dx)\vol(dy)\\
        &\le \int_{\XX^d\times\XX^d} \PP(D_{x,y} F_t^{(\alpha)} \ne 0)^{1/2} (\EE(D_{x,y} F_t^{(\alpha)})^4)^{1/2} \,\vol(dx)\vol(dy).
    \end{align*}
    By Lemma~\ref{lem:DxBounds}, this can be bounded by
    \begin{multline*}
        C \int_{B_t} \int_{\XX^d} \exp(-c \rho(d(x,y)/2)) \,\vol(dx)\vol(dy)\\
        + C \int_{\XX^d \setminus B_t} \int_{\XX^d} \exp(-c \rho(d(x,\partial B_t)/2))\exp(-c \rho(d(x,y)/2)) \,\vol(dx)\vol(dy)
    \end{multline*}
    for some constants $C,c>0$ depending only on $\alpha$ and the choice of $\XX^d$. Using standard spherical coordinates in the Euclidean case and the polar integration formula \eqref{eq:PolIntHypGeom} for the hyperbolic space, {it} is not hard to see that
    \[
    \sup_{x \in \XX^d} \int_{\XX^d} \exp(-c \rho(d(x,y)/2)) \,\vol(dy) < \infty
    \]
    and hence we are left to bound
    \begin{equation}\label{eq:P2.1}
    C \vol(B_t) + C \int_{\XX^d \setminus B_t} \exp(-c \rho(d(x,\partial B_t)/2)) \,\vol(dx).
    \end{equation}
    
    In the Euclidean case, $\rho(t)=\vol(B_t)$ and the second term, after changing to polar coordinates, can be written as
    \begin{align*}
        C d\kappa_d \int_t^\infty r^{d-1} \exp(-c \kappa_d2^{-d} (r-t)^d) \,dr 
        &= C d\kappa_d \int_0^\infty (u+t)^{d-1} \exp(-c \kappa_d2^{-d} u^d)\, du\\
        &\leq C d\kappa_d \int_0^\infty 2^{d-1} (u^{d-1} +t^{d-1}) \exp(-c \kappa_d2^{-d} u^d) \,du \\
        &= O(t^{d-1}).
    \end{align*}
    Hence there is a constant $C_1>0$ depending on $\alpha$ (and $\RR^d$) only such that
    \[
    \int_{\XX^d\times\XX^d} \EE (D_{x,y} F_t^{(\alpha)})^2 \,\vol(dx)\vol( dy) \leq C_1\kappa_d t^d = C_1 \rho(t)
    \]
    for $\XX^d=\RR^d$.
    
    In the hyperbolic case, by \eqref{eq:hypvol}, one has that
    \[
    \vol(B_t) \leq \frac{\Gamma_d}{\gamma_d}\rho(t).
    \]
    For the second term of \eqref{eq:P2.1}, it holds that for $x \notin B(p,t)$, one has $d(x,B(p,t)) = d(x,p)-t$. By changing to hyperbolic polar coordinates using \eqref{eq:PolIntHypGeom}, as done in the proof of Lemma~\ref{lem:exptails3}, we deduce this term is bounded by
    \[
    Cd\kappa_d \int_{S^{d-1}(p)} \int_t^\infty \sinh(r)^{d-1} \exp(-c\rho((r-t)/2)) \,dr\sigma_p(dv).
    \]
    Using again that $\sinh(x) \leq e^x$ for $x>0$, and introducing the change of coordinates $u = r-t$, it follows that the above is bounded by
    \[
    Cd\kappa_d \int_0^\infty \exp((d-1)(u+t)) \exp(-c\rho(u/2))\, du = \frac{Cd\kappa_d}{\gamma_d} \rho(t) \int_0^\infty e^{(d-1)u-c\rho(u/2)} \,du \le C_2 \rho(t)
    \]
    for some constant $C_2>0$ depending on {$\alpha$} (and $\HH^d$) only. Thus, we conclude that in the hyperbolic case $\XX^d=\HH^d$,
    \[
    \int_{\XX^d\times\XX^d} \EE (D_{x,y} F_t^{(\alpha)})^2 \,\vol(dx)\vol( dy) \leq C_3 \rho(t),
    \]
    for some constant $C_3>0$ which depends on $\alpha$ and $\HH^d$. 
    
    It remains to show a lower bound of order $\rho(t)$ for the quantity
    \[
    \int_{\XX^d} \EE (D_x F_t^{(\alpha)})^2 \,\vol(dx).
    \]
    Let $x \in B_t$ such that $B(x,1) \subset B_t$. Recalling the notation established in \eqref{eq:S(z,v,r)}, we start by \textbf{claiming} that there exists an $\epsilon>0$ small enough such that for any $y \notin B(x,1-\epsilon)$ and any $v \in S^{d-1}(y) \cap H_{\exp_y^{-1}(x)}(y)$, the set $S(y,v,d(x,y))$ contains a subset of the form $D \cap B(x,1) \setminus B(x,1-\epsilon)$, where $D=\exp_x(C)$ for some cone $C \subset T_x\XX^d$, centred at $x$ and of angular radius $\beta$, where $\beta$ depends only on $\epsilon$. This is represented by the red area in Figure~\ref{fig:VarF}. To \textbf{prove the claim}, define $x':= \exp_y^{-1}(x)$ and denote by $\mathcal{H}_{x'}(y)$ the hyperplane in $T_y\XX^d$ passing through $y$ and {has} normal vector $x'-y$. Let $v'$ be the projection of $v$ onto $\mathcal{H}_{x'}(y)$. Then it holds that $H_{v'}(y) \cap H_{x'}(y) \subset H_v(y) \cap H_{x'}(y)$. Indeed, one has
    \[
    v' = v - \frac{\langle x'-y, v-y \rangle}{\lVert x'-y \rVert^2} (x'-y).
    \] Hence for any $w \in H_{v'}(y) \cap H_{x'}(y)$, we have
    \[
    \langle v-y,w-y \rangle = \underbrace{\langle v'-y, w-y \rangle}_{\geq 0 \text{, as } w \in H_{v'}(y)} + \frac{1}{\lVert x'-y \rVert^2} \underbrace{\langle x'-y, v-y \rangle}_{\geq 0 \text{, as } v \in H_{x'}(y)} \cdot \underbrace{\langle x'-y, w-y \rangle}_{\geq 0 \text{, as } w \in H_{x'}(y)} \geq 0,
    \]
    therefore $w \in H_v(y)$.  Next, let $y'$ be the projection of $y$ onto $\partial B(x,1-\epsilon)$ for some small $\varepsilon$. Since $\exp_y^{-1}(y') \in H_{x'}(y)$, it now holds that
    \[
    \exp_y(\exp_y^{-1}(y') + H_{v'}(y) \cap H_{x'}(y)) \subset \exp_y(H_{v'}(y) \cap H_{x'}(y)) \subset \exp_y(H_v(y)).
    \]
    Moreover, $B(y',d(x,y')) \subset B(y,d(x,y))$ and we note that, due to the fact that the hyperplanes $\mathcal{H}_{x'}(y)$ and $\mathcal{H}_{v'}(y)$ are perpendicular, the set
    \[
    \exp_y(\exp_y^{-1}(y') + H_{v'}(y) \cap H_{x'}(y)) \cap B(y',d(x,y')) \cap B(x,1) \setminus B(x,1-\epsilon)
    \]
    is always the same up to isometry, independently of the choice of $y$ and $v$. One can now choose $\epsilon>0$ small enough such that there is a cone $C\subset T_x\XX^d$ with opening angle $\beta>0$ depending on $\epsilon$ such that for $D:=\exp_x(C)$, the intersection $D \cap B(x,1)\setminus B(x,1-\epsilon)$ is included in this set. This proves the claim.

    We choose $\epsilon,\beta>0$ as above. Cover $T_x\XX^d$ with $K=K_{d,\epsilon}$ cones $C_1,\hdots,C_K$ of angular radius $\beta/2$ with apex at $x$.  Choose $K$ points $z_1,\ldots,z_K$  such that  $1-\epsilon < d(x,z_i) < 1$ and $z_i$ lies within the projection $D_i:=\exp_x(C_i)$ of cone $C_i$. We \textbf{claim} that for each point $y$ outside $B(x,1-\epsilon)$, the {set} $S(y,\exp_y^{-1}(x),1)$ is completely covered by the union $\bigcup_{0<d(y,z_i)<d(x,y)}S(y,\exp_y^{-1}(z_i),1)$. To \textbf{prove the claim}, assume this is not the case. Then there is a point $y$ outside $B(x,1-\epsilon)$ and a point $v \in S^{d-1}(y)$ such that the interior of $S(y,v,d(x,y))$ contains non of the points $z_1,\hdots,z_K$. However, by the claim we proved earlier in this proof, the set $S(y,v,d(x,y))$ always contains a set of the form $D \cap B(x,1) \setminus B(x,1-\epsilon)$, where $D=\exp_x(C)$ and $C \subset T_x\XX^d$ is a cone of apex $x$ and angular radius $\beta$. The cone $C$ must contain one of the cones $C_i$, and hence $D_i \cap B(x,1) \setminus B(x,1-\epsilon)$ is included inside of $S(y,v,d(x,y))$.

    Note that the above discussion continues to hold if we disturb the points $z_1,\hdots,z_K$ by some $\epsilon'>0$ chosen small enough.

    Assume now that the point set $\eta$ is such that $B(x,1-\epsilon)$ is empty, $\eta$ contains exactly one point $z_i'$ in each of $B(z_1,\epsilon'),\hdots,B(z_K,\epsilon')$ and no other points of $\eta$ are within $B(x,1)$.
    
    In this situation, points outside $B(x,1-\epsilon)$ (including $z_1',\hdots,z_K'$) might gain edges, but will not lose any if $x$ is added. Upon the addition of $x$, at least $d+1$ edges are added from $x$ to some of the points $z_i$, since $x$ connects to its neighbours until the convex hull is completed. As $d(x,z_i)>1-\epsilon$, we conclude that
    \[
    |D_x F_t^{(\alpha)}| \ge (d+1)(1-\epsilon)^{\alpha} \mathbbm{1}\Big\{\eta(B(z_1,\epsilon')) = \ldots = \eta(B(z_K,\epsilon'))=1, \eta\Big(B(x,1) \setminus \bigcup_{i=1}^KB(z_i,\epsilon')\Big) = 0\Big\}.
    \]
    It follows that $\EE(D_x F_t^{(\alpha)})^2 \ge c > 0$ for some constant $c$ depending, by isometry invariance of the volume, only on $\alpha$ and the choice of $\epsilon$, $\epsilon'$ and $\XX^d$, and hence,
    \[
    \int_{\XX^d} \EE(D_xF_t^{(\alpha)})^2 \,\vol(dx) \ge \int_{B_{t-1}} \EE(D_xF_t^{(\alpha)})^2\,\vol( dx) \ge c \vol(B_{t-1}) \ge c' \rho(t)
    \]
    for another constant $c'$ only depending on the same parameters. In particular, this proves \eqref{CondST} for a suitable constant $c > 0$ only depending on $\alpha$ and the choice of $\epsilon$ and $\XX^d$, so that we obtain 
    \[
    \Var F_t^{(\alpha)} \ge c \rho(t).
    \]
    This completes the argument.
    \end{proof}

    \begin{figure}
        \centering
        \begin{subfigure}[t]{.4\textwidth}
            \centering
            \includegraphics[width=\textwidth]{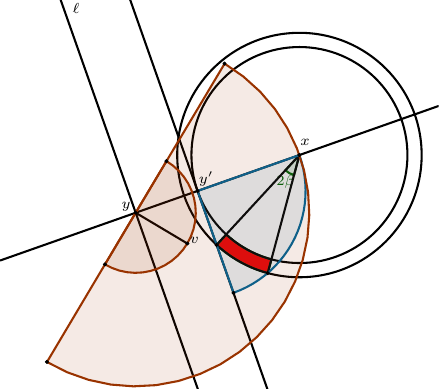}
            \caption{The construction for the variance in the Euclidean case. Note that for any $v$ on the boundary of $S(y,x,1)$, the {set} $S(y,v,d(x,y))$ always contains the blue area quarter-sphere. The blue area will always contain a set of the form $C \cap B(x,1) \setminus B(x,1-\epsilon)$, represented by the red area.}
            \label{fig:EuclVarF}
        \end{subfigure}%
        \hspace{.099\textwidth}
        \begin{subfigure}[t]{.4\textwidth}
            \centering
            \includegraphics[width=\textwidth]{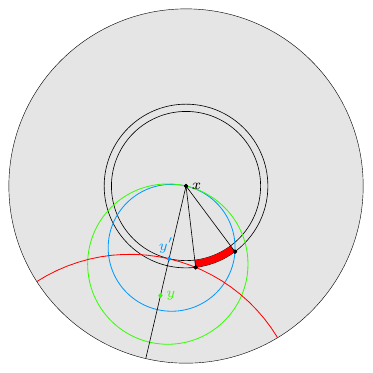}
            \caption{The construction in hyperbolic space, represented in the Poincaré disk model. Given any green point $y$ outside $B(x,1-\epsilon)$, we can again find a red area determined by a certain angle at the origin.}
            \label{fig:HypVarF}
        \end{subfigure}
        \caption{The construction of the variance lower bound for $F^{(\alpha)}_t$, in both the Euclidean and the hyperbolic case.}
        \label{fig:VarF}
    \end{figure}
	
After these preparations, the proof of Theorem \ref{thm:main} is now easily completed.

\begin{proof}[Proof of Theorem \ref{thm:main}]
    We apply Proposition~\ref{thm:LPS6.1}, which is justified by Lemma~\ref{lem:DxBounds}. Since by Lemma~\ref{lem:VarBd} we know that
    \[
    \Var F_t^{(\alpha)} \geq c_1 \rho(t),
    \]
    for $c_1>0$ only depending on $\alpha$ and the choice of $\XX^d$, it suffices to show that there is a constant $c_2>0$ depending only on $\alpha$ and the choice of $\XX^d$ such that 
    \begin{align*}
        &\int_{\XX^d} \left( \int_{\XX^d} \PP(D_{x,y} F_t^{(\alpha)} \neq 0)^{1/20} \,\vol(dy) \right)^2 \vol(dx),\\
        &\int_{\XX^d} \PP(D_x F_t^{(\alpha)} \neq 0)^\beta \,\vol(dx),\\
        & \int_{\XX^d} \int_{\XX^d} \PP(D_{x,y} F_t^{(\alpha)} \neq 0)^{1/20} \,\vol(dy) \vol(dx)
    \end{align*}
    are all bounded by $c_2\rho(t)$ for every $\beta>0$. This follows by computations akin to the ones done in the first part of the proof of Lemma~\ref{lem:VarBd}. Finally, we may clearly replace $\rho(t)$ by $\vol(B_t)$ by definition in the Euclidean case and by \eqref{eq:hypvol} in the hyperbolic case.
\end{proof}

\subsection{Proof of Theorem \ref{thm:main2}}

Large parts of the proof of Theorem \ref{thm:main2} are very similar to the proof of Theorem \ref{thm:main}. The only major difference lies in the proof of an analogue of Lemma \ref{lem:VarBd}, i.e., a lower variance bound for $G_t^{(k)}$ based on Proposition \ref{thm:ST1.1}.

\begin{lemma}\label{lem:VarBd2}
    For $t \geq 2$ it holds that
    \[
    \Var G_t^{(k)} \ge c \rho(t),
    \]
    where $c>0$ is some constant depending on $k$ and the choice of $\XX^d$ only.
\end{lemma}

\begin{proof}
    \begin{figure}[t]
        \centering
        \includegraphics[width=.8\textwidth]{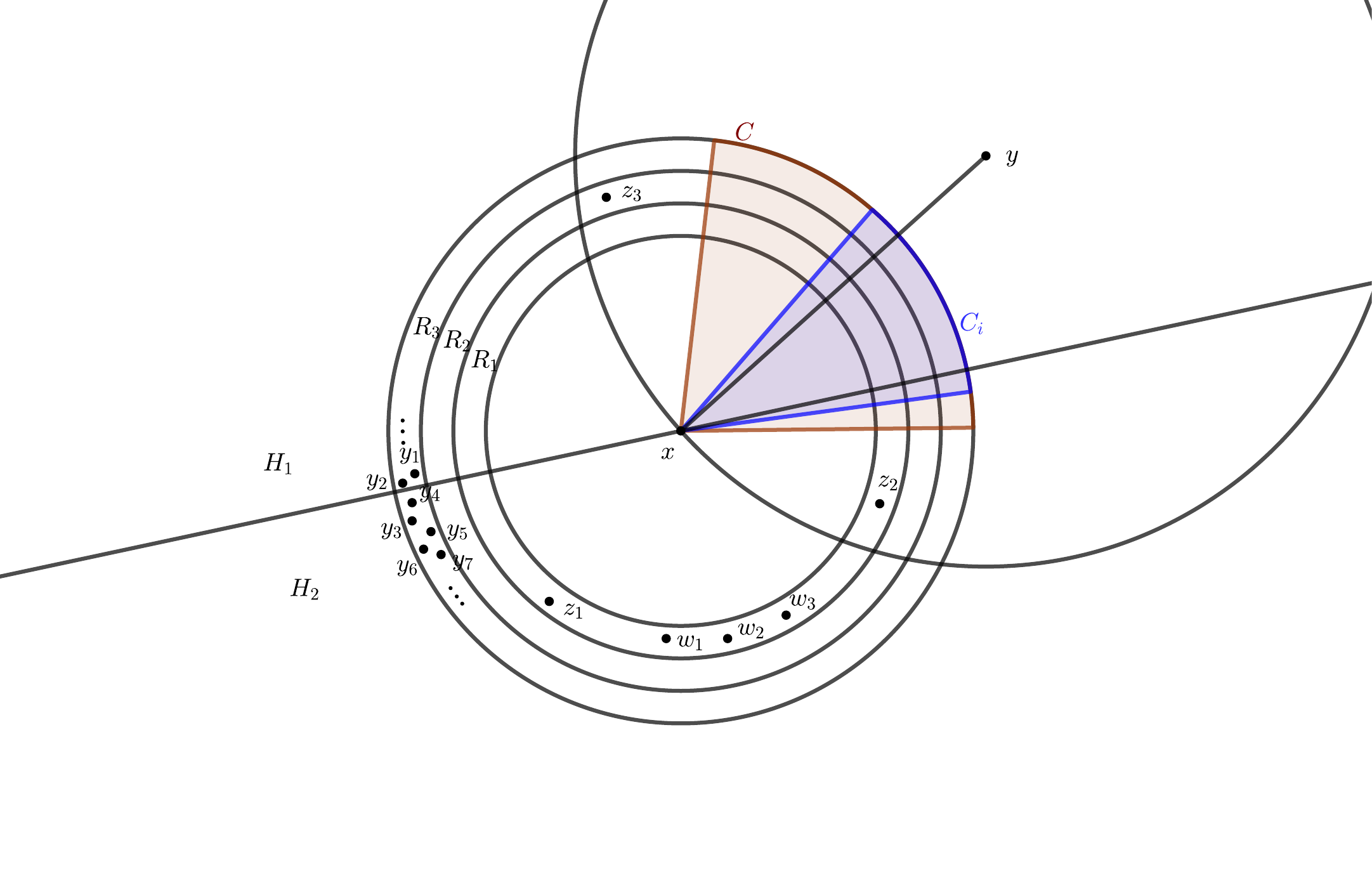}
        \caption{Illustration of the construction in the Euclidean case. Here $k=6$.}
        \label{fig:DegVarEucl}
    \end{figure}
    Let $x \in \XX^d$ such that $B(x,1) \subset B_t$. We construct three rings around $x$ by defining
    \begin{align*}
        R_1 &:= \Big\{y \in \XX^d \colon \frac{2}{3} < d(x,y) < \frac{2}{3} + \frac{1}{9}\Big\},\\
        R_2 &:= \Big\{y \in \XX^d \colon \frac{2}{3} + \frac{1}{9} < d(x,y) < \frac{2}{3} + \frac{2}{9}\Big\},\\
        R_3 &:= \Big\{y \in \XX^d \colon \frac{2}{3} + \frac{2}{9} < d(x,y) < 1\Big\}.
    \end{align*}
    See Figure~\ref{fig:DegVarEucl} for an illustration. Split $\XX^d$ into two half-spaces $H_1$, $H_2$ such that $x$ lies on the boundary between the two. Now assume we have points as follows:
    \begin{enumerate}[(1)]
        \item Let $z_1, \ldots, z_d \in R_1 \cap H_2$ and $z_{d+1} \in R_2 \cap H_1$ such that $x$ is included in the convex hull of $z_1, \ldots, z_{d+1}$.\label{item:convhull}
        \item Let $w_1, \ldots, w_{k-d-1} \in R_1 \cap H_2$. 
        \item Let $y_1, \ldots, y_K \in R_3$ such that for every $y \in \XX^d \setminus B(x,\frac{2}{3})$, there are at least $k+2$ points among $y_1, \ldots, y_K$ which are closer to $y$ than $x$.\label{item:k+2nbrs}
    \end{enumerate}
    See again Figure~\ref{fig:DegVarEucl} for an illustration.

    The validity of the construction \eqref{item:convhull} is equivalent to choosing $z_1,\hdots,z_{d+1}$ with $z_1, \ldots, z_d \in R_1 \cap H_2$ and $z_{d+1} \in R_2 \cap H_1$ such that for every $v \in S^{d-1}(x)$, the set $S(x,v,8/9)$ is non-empty. This can be achieved by constructing cones in the tangent space and projecting them down to $\XX^d$, making sure that every projection contains a point. To see \eqref{item:k+2nbrs}, we argue as follows. 
    Let $y \in \XX^d \setminus B(x,\frac{2}{3})$ and define $y'$ as the projection of $y$ onto $\partial B(x,\frac{2}{3})$. Then $B(y',d(x,y')) \subset B(y,d(x,y))$ and the set $B(y',d(x,y')) \cap B(x,1)$ is, up to isometry, the same irrespective of the choice of $y$. Moreover, $B(y',d(x,y')) \cap B(x,1)$ contains a set of the form $D \cap B(x,1)$, where $D=\exp_x(C)$ for some cone $C \subset T_x\XX^d$ with apex at $x$ and a given angular radius $\beta>0$. Now, cover $T_x\XX^d$ with cones $C_1, \ldots, C_K$ with apex at $x$ and angular radius $\frac{\beta}{2}$ and put $k+2$ points $y_i, \ldots, y_{i+k+1}$ into $\exp_x(C_i) \cap R_3$. There is now a cone $C_i$ such that $\exp_x^{-1}(y) \in C_i$. Moreover, $C_i \subset C$, and so $\exp_x(C_i) \cap R_3 \subset D \cap B(y,d(x,y))$. Hence, all points $y_i, \ldots, y_{i+k+1}$ are closer to $y$ than $x$. 
    
    Assume now that $\eta$ contains the points $z_1, \ldots, z_{d+1}, w_1, \ldots, w_{k-d-1}, y_1, \ldots, y_K$ and no other points within $B(x,1)$. Then, upon the addition of $x$, the following holds: first, $\overline{\mathrm{deg}}(x)=k$. Indeed, the half-ball $H_1 \cap B(x,\frac{2}{3} + \frac{1}{9})$ does not contain any point of $\eta$, and the half-ball $H_1 \cap B(x,\frac{2}{3} + \frac{2}{9})$ only contains the point $z_{d+1}$. Hence, $x$ must connect to $z_{d+1}$. Since $d(x,z_i) < \frac{2}{3} + \frac{1}{9}$ for all $i = 1, \ldots, d$ and $d(x,w_i) < \frac{2}{3} + \frac{1}{9}$ for all $i=1, \ldots, k-d-1$, the point $x$ must also connect to $z_1, \ldots, z_d, w_1, \ldots, w_{k-d-1}$. The placement of the points $z_1, \ldots, z_{d+1}$ ensures that the convex hull is complete, and hence, $x$ will not connect to any point outside of $B(x,\frac{2}{3} + \frac{2}{9})$.

    Moreover, any point $y \in \eta$ such that $\overline{\mathrm{deg}}(y,\eta) \ne \overline{\mathrm{deg}}(y,\eta+\delta_x)$ is such that $\overline{\mathrm{deg}}(y,\eta) > k$ and $\overline{\mathrm{deg}}(y,\eta+\delta_x) > k$. To see this, first note that $\overline{\mathrm{deg}}(y)$ can only change with the addition of $x$ if $y$ connects to $x$. In this case, $y$ must also connect to any point of $\eta$ within $B(y,d(x,y))$. Since $\eta \cap B(x,\frac{2}{3}) = \varnothing$, we have $d(x,y) > \frac{2}{3}$ and hence by construction, $B(y,d(x,y))$ contains at least $k+2$ of the points $y_1, \ldots, y_K$. It follows that $\overline{\mathrm{deg}}(y,\eta+\delta_x) \ge k+1$. Moreover, $y$ must have been connected to all points of $\eta$ within $B(y,d(x,y))$ prior to the addition of $x$ (otherwise $y$ would not connect to $x$), and hence $\overline{\mathrm{deg}}(y,\eta) \ge k+1$.

    Altogether, under the above assumption on $\eta$ we have that
    \begin{align*}
        D_x G_t^{(k)} &= \ind{\overline{\mathrm{deg}}(x)=k} + \sum_{y\in \eta|_{B_t}} \ind{\overline{\mathrm{deg}}(y,\eta+\delta_x)=k} \ind{\overline{\mathrm{deg}}(y,\eta)\ne k}\\
        &\phantom{=} - \sum_{y\in \eta|_{B_t}} \ind{\overline{\mathrm{deg}}(y,\eta+\delta_x)\ne k} \ind{\overline{\mathrm{deg}}(y,\eta)= k}\\
        &= 1 + 0 - 0 = 1.
    \end{align*}
    Moreover, the arguments given above remain true when every point of $z_1, \ldots, z_{d+1},w_1, \ldots, w_{k-d-1},y_1,\linebreak[2] \ldots, y_K$ is slightly perturbed by distance at most $\epsilon > 0$, where $\epsilon$ is small enough. Hence, on the event of positive probability
    \begin{align*}
        \{ &\eta(B(z_1,\epsilon))=1, \ldots, \eta(B(z_{d+1},\epsilon))=1, \eta(B(w_1,\epsilon))=1, \ldots, \eta(B(w_{k-d-1},\epsilon))=1,\\
        &\eta(B(y_1,\epsilon))=1, \ldots, \eta(B(y_K,\epsilon))=1, \eta(B(x,1))=k+K\},
    \end{align*}
    we have $D_xG_t^{(k)} = 1$ and thus, $\EE(D_x G_t^{(k)})^2 \ge c > 0$ for some constant $c$ only depending on $k$ and the choice of $\epsilon$ and $\XX^d$. The remaining parts of the proof follow from arguments similar to the ones used in the proof of Lemma \ref{lem:VarBd}.
\end{proof}

\begin{proof}[Proof of Theorem \ref{thm:main2}]
    To begin, recall that $R(x,\eta)$ from Lemma \ref{lem:DOloc} is a radius of stabilization for $G_t^{(k)}$. Moreover, note that
    \begin{align*}
    D_xG_t^{(k)} &= \ind{\overline{\mathrm{deg}}(x)=k}\ind{x \in B_t} + \sum_{y\in \eta|_{B_t}} \ind{\overline{\mathrm{deg}}(y,\eta+\delta_x)=k} \ind{\overline{\mathrm{deg}}(y,\eta)\ne k} - \sum_{y\in \eta|_{B_t}} \ind{\overline{\mathrm{deg}}(y,\eta+\delta_x)\ne k} \ind{\overline{\mathrm{deg}}(y,\eta)= k}
    \end{align*}
    and hence,
    \begin{align*}
        |D_xG_t^{(k)}| &\le 1 + \sum_{y\in \eta|_{B_t}} \ind{\overline{\mathrm{deg}}(y,\eta) \ne \overline{\mathrm{deg}}(y,\eta+\delta_x)}\le 1 + \eta(B(x,R(x,\eta))).
    \end{align*}
    In particular, applying Lemma \ref{lem:etamom} yields
    \[
    \EE \left| D_x G_t^{(k)} \right|^r \le C
    \]
    for any $r > 0$, where $C>0$ is a constant which depends on $r$ and the choice of $\XX^d$ only.

    Similarly, we have
    \begin{align*}
        |D_{x,y} G_t^{(k)}| &\le |D_x G_t^{(k)}(\eta)| + |D_x G_t^{(k)}(\eta+\delta_y)|\\
        &\le 1 + \eta(B(x,R(x,\eta)) + 1 + 1 + \eta(B(x,R(x,\eta+\delta_y)).
    \end{align*}
    This shows that we also have
    \[
    \EE \left| D_{x,y} G_t^{(k)} \right|^r \le C
    \]
    for any $r > 0$, where $C>0$ is a constant which depends on $r$ and the choice of $\XX^d$ only. Furthermore, the probabilities of $D_xG_t^{(k)} \ne 0$ and $D_{x,y}G_t^{(k)} \ne 0$ are bounded akin to the case of the edge-sums discussed in Lemma \ref{lem:DxBounds}, as the reasoning in the proof of the latter is solely based on the behaviour of the radii of stabilization. Together with Lemma \ref{lem:VarBd2}, we may therefore complete the proof {of the upper bound for both $\diamondsuit=W$ and $\diamondsuit=K$} in the same way as the proof of Theorem \ref{thm:main}.
    
    {To deduce a corresponding lower bound on the rate of convergence with $\diamondsuit=K$ we observe that $G_t^{(k)}$ is a square-integrable integer-valued random variable. It follows from an argument of Englund \cite{Englund} (see also \cite[Section 3.3]{RednossDiss}) that for such random variables the rate of convergence measured with respect to the Kolmogorov distance is always bounded by a constant multiple of the minimum of one and $1$ divided by the square-root of the variance. The lower bound thus follows from this in conjection with Lemma \ref{lem:VarBd2} and the definition of $\rho(t)$.}
\end{proof}

\subsection*{Acknowledgements}

Parts of this paper were written when the authors were participants of the Dual Trimester Program \textit{Synergies between modern probability, geometric analysis and stochastic geometry} at the Hausdorff Research Institute for Mathematics, Bonn. All support is gratefully acknowledged.

H.S.\ was supported by the Deutsche Forschungsgemeinschaft (DFG) via CRC 1283 \emph{Taming uncertainty and profiting from randomness and low regularity in analysis, stochastics and their applications}. T.T.\ was supported by the Luxembourg National Research Fund (PRIDE17/1224660/GPS) and by the UK Engineering and Physical Sciences Research Council (EPSRC) grant (EP/T018445/1). C.T.\ was supported by the Deutsche Forschungsgemeinschaft (DFG) via SPP 2265 \textit{Random Geometric Systems} and via SPP 2458 \textit{Combinatorial Synergies}.

\end{document}